\documentclass[a4paper,11pt,oneside]{article}
\usepackage{amsmath,amsthm,amsfonts,amssymb,amscd}
\usepackage{authblk}
\usepackage{mathtools}
\usepackage{pgf,tikz}
\usepackage{mathrsfs}
\usepackage{float}
\usepackage{csvsimple}
\usepackage{enumerate}
\usepackage{booktabs}
\usepackage{pgfplots}
\usepackage[noend]{algpseudocode}
\usepackage{graphicx,bm,xcolor}
\usepackage{algorithm}
\usepackage{pdfpages}
\usepackage{algpseudocode}
\usepackage{stmaryrd}
\usetikzlibrary{arrows}
\usepackage{caption}

\usepackage{t1enc}
\usepackage[german,english]{babel}
\usepackage{verbatim} 

\usepackage{url}

\newif\ifMAKEPICS
\MAKEPICSfalse

\ifMAKEPICS
\usepackage[cleanup,subfolder]{gnuplottex}
\usepackage{xparse}

\ExplSyntaxOn
\DeclareExpandableDocumentCommand{\convertlen}{ O{cm} m }
{
	\dim_to_decimal_in_unit:nn { #2 } { 1 #1 } cm
}
\ExplSyntaxOff
\fi

\makeindex      

\newcommand{\ignore}[1]{}

\newtheorem{theorem}{Theorem}[section]

\newtheorem{remark}[theorem]{Remark}

\begin{document}

\title{Hierarchical {DWR} {E}rror {E}stimates for the {N}avier {S}tokes 
{E}quation:  $h$ and $p$ {E}nrichment}
	\author[1]{B. Endtmayer}
		\author[1]{U. Langer}
	\author [2]{J. P. Thiele}
	\author[2]{T. Wick}
	\affil[1]{Johann Radon Institute for Computational and Applied Mathematics, Austrian Academy of Sciences, Altenbergerstr. 69, A-4040 Linz, Austria}
	%\affil[3] ENGLISH !!!
	\affil[2]{Institut f\"ur Angewandte Mathematik, Leibniz Universit\"at Hannover, Welfengarten 1, 30167 Hannover, Germany}
	
	\date{}
	
	\maketitle

\abstract{
In this work, we further develop multigoal-oriented a posteriori error estimation for 
the nonlinear, stationary, incompressible Navier-Stokes equations. It is an extension 
of our previous work 
[B. Endtmayer, U. Langer, T. Wick:  Two-side a posteriori error estimates for the {DWR} method, {\it SISC}, 2019, accepted].
We now focus on $h$ mesh refinement and
$p$ enrichment for the error estimator. 
These advancements are demonstrated with 
the help of a numerical example%
.
}

%%%%%%%%%%%%%%%%%%%%%%%%%%%%%%%%%%%%%%%%%%%%%%%%%%%%%%%%%%%%%%%%%%
\section{Introduction}
\label{sec: Introduction}
Multigoal-oriented error estimation offers the opportunity to control several 
quantities of interest simultaneously. In recent years, we have developed 
a version \cite{EnLaWi18,2019Two} 
%based on 
which relies on
the dual-weighted residual method \cite{BeRa01},
and 
%which 
also balances the discretization error with the nonlinear iteration error \cite{RanVi2013}.
The localization is based on the weak formulation proposed in \cite{RiWi15_dwr}.
Our method 
%is based 
uses on 
hierarchical finite element spaces. 
Here, we investigate $h$ refinement along with $p$ refinement 
%in the enriched spaces. 
to generate enriched spaces. 
These ideas are applied to the stationary 
incompressible Navier-Stokes equations. 
%\red{Here}, 
It is well-known that the 
spaces for the velocities and the pressure 
%variables 
must be balanced 
in order to satisfy an inf-sup condition \cite{GiRa1986}. 
These requirements must be reflected 
in the design of the adjoint problems in 
dual-weighted residual error estimation 
and our proposed $p$ refinement. To demonstrate the performance of the 
error estimator, we adopt the 2D-1 fluid flow 
benchmark \cite{TurSchabenchmark1996}.

%%%%%%%%%%%%%%%%%%%%%%%%%%%%%%%%%%%%%%%%%%%%%%%%%%%%%%%%%%%%%%%%%%
\section{The Model Problem and Discretization}
\label{sec:Model Problem and Discretization}
\subsection{The Model Problem} \label{SubSection: Model Problem}
We consider the  stationary Navier Stokes 2D-1 benchmark  problem \cite{TurSchabenchmark1996} 
as our model problem. This configuration  was also considered in \cite{2019Two}.
The domain $\Omega \subset \mathbb{R}^2$ is given by 
$(0,2.2) \times (0,H) \setminus B$, 
%and $B$ is given by 
and $B$ is
the ball with the center $(0.2,0.2)$ and the radius $0.05$ as given in \cite{TurSchabenchmark1996} and visualized in Figure~\ref{Figure: Omega}.
Find $\vec{u} = (u,p)$ such that
\begin{eqnarray*}
- \text{div}(\nu(\nabla u + \nabla u^T))  + (u\cdot \nabla ) u - \nabla p =&  0,  &\qquad   \text{ in }\Omega, \\ 
 -\text{div}(u) =& 0,& \qquad    \text{ in }\Omega, \\
 u=& u_{\text{in}}& \qquad \text{ on }	\Gamma_{\text{in}},  \\
 u=&0 & \qquad \text{ on }\Gamma_{\text{no-slip}},\\
 \nu(\nabla u+\nabla u ^T)\vec{n}+p\vec{n}=&0&  \qquad \text{ on }\Gamma_{\text{out}},
\end{eqnarray*}
where $\Gamma_{\text{in}} :=\{x=0\} \cap \partial \Omega$, 
				$\Gamma_{\text{no-slip}} := \overline{\partial \Omega \setminus (\Gamma_{\text{in}} \cup \Gamma_{\text{out}})}$
			and $\Gamma_{\text{out}} :=(\{x=2.2\} \cap \partial \Omega)\setminus \partial (\{x=2.2\} \cap \partial \Omega).$
Furthermore, the viscosity $v = 10^{-3}$  and $u_{\text{in}}(x,y)=  (0.3 w(y),0)$ with $w(y)=4y(H-y)/H^2$ and $H=0.41$. 
%We notice that $u_\text{in}=0$ on $\Gamma_{\text{in}}\cap  \Gamma_{\text{no-slip}}$ and therefore it fulfills the compatibility conditions.
The corresponding weak form 
%is given by:
reads as follows:
Find $\vec{u} = (u,p) \in V_{BC} := [H^1(\Omega)]_{BC}^2 \times L^2(\Omega)$ 
such that 
\begin{equation} \label{Equation: Model Problem}
A(\vec{u})(\vec{v}) = 0 \quad\forall \vec{v} = (v_u,v_p) \in V_{0} := [H^1_{0}(\Omega)]^2 \times L^2(\Omega)
\end{equation}
with
\begin{eqnarray*}
A(\vec{u})(\vec{v}):=&(\nu(\nabla u + \nabla u^T), \nabla v_u)_{[L^2(\Omega)]^{2 \times 2}}+ ((u\cdot \nabla ) u,v_u)_{[L^2(\Omega)]^{2}} \nonumber
\\ &+ (p,\text{div}(v_u))_{L^2(\Omega)}-(\text{div}(u), v_p)_{L^2(\Omega)},\nonumber
\end{eqnarray*}
where $[H^1(\Omega)]_{BC}^2:= \{ u \in [H^1(\Omega)]^2: u_{|\Gamma_{\text{in}}}  = u_{\text{in}} \wedge u_{|\Gamma_{\text{no-slip}}} =0  \}$ and \\$[H^1_{0}(\Omega)]^2:= \{ v \in [H^1(\Omega)]^2: v_{|\Gamma_{\text{in}}}  = 0 \wedge v_{|\Gamma_{\text{no-slip}}}  =0\}$.
\begin{figure}[H]
\begin{minipage}[t]{0.47\linewidth}
\definecolor{ffffff}{rgb}{1,1,1}
\definecolor{ffqqqq}{rgb}{1,0,0}
\definecolor{ttqqqq}{rgb}{0.2,0,0}
\definecolor{qqzzcc}{rgb}{0,0.6,0.8}
\definecolor{zzttqq}{rgb}{0.6,0.2,0}
\begin{tikzpicture}[line cap=round,line join=round,>=triangle 45,x=2.3cm,y=2.3cm]
\clip(-0.2,-0.0) rectangle (2.4,0.41);
\fill[line width=2pt,color=zzttqq,fill=zzttqq,fill opacity=0.10000000149011612] (0,0) -- (0,0.41303969424395254) -- (2.208741109888243,0.41437055494412134) -- (2.2008824918026844,0) -- cycle;
\draw [line width=0.5pt,color=black,fill=ffffff,fill opacity=1] (0.2,0.2) circle (0.05);
\draw [line width=1pt,color=qqzzcc] (0,0)-- (0,0.41303969424395254);
\draw [line width=1pt,color=ttqqqq] (0,0.41303969424395254)-- (2.208741109888243,0.41437055494412134);
\draw [line width=1pt,color=ffqqqq] (2.208741109888243,0.41437055494412134)-- (2.2008824918026844,0);
\draw [line width=1pt,color=ttqqqq] (2.2008824918026844,0)-- (0,0);
\begin{scriptsize}
\draw[color=qqzzcc] (-0.1,0.205) node {$\Gamma_{\text{in}}$};
\draw[color=ffqqqq] (2.08,0.205) node {$\Gamma_{\text{out}}$};
\draw[color=ttqqqq] (1.1157337774147214,0.1096713051094704) node {$\Gamma_{\text{no-slip}}$};
\end{scriptsize}
\end{tikzpicture}
\end{minipage} \hfill
\begin{minipage}[t]{0.47\linewidth}
\includegraphics[scale=0.234]{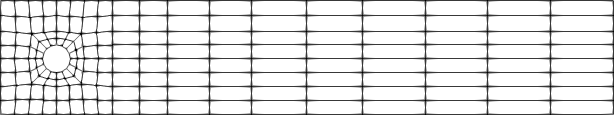}
\end{minipage} 
\caption{ The computational domain $\Omega$  (left) and the initial mesh (right). \label{Figure: Omega}}
%\caption{ The domain $\Omega$ with the given conditions (left) and the initial mesh (right) \label{Figure: Omega}}
\end{figure}
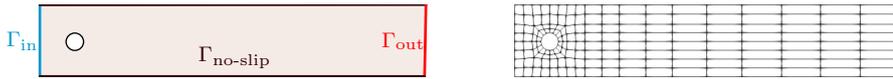

\subsection{Discretization}
Let $\mathcal{T}_h$ be a decomposition of $\Omega \subset \mathbb {R}^2$ into quadrilateral elements. Furthermore, we assume that $\mathcal{T}_{\frac{h}{2}}$ is the uniform refinement of $\mathcal{T}_h$.
We discretize our problem using $[Q^2_c]^2$ , i.e. piecewise bi-quadratic elements for the velocity $u$, and $Q^1_c$, i.e. piecewise bi-linear elements for the pressure $p$. 
The resulting space using the mesh $\mathcal{T}_h$ will be denoted by $V_h$. For a more detailed explanation of the discretization, we refer to \cite{2019Two}.
The resulting space using the mesh $\mathcal{T}_{\frac{h}{2}}$ will be denoted by $V_{\frac{h}{2}}$. We say $V_{\frac{h}{2}}$ is the (hierarchical) $h$-refined finite element space of $V_h$. 
Furthermore we consider using $[Q^4_c]^2$ , i.e. piecewise bi-quartic elements for the velocity $u$, and $Q^2_c$, i.e. piecewise bi-quadratic elements for the pressure $p$. The resulting finite element space using the mesh $\mathcal{T}_h$ will be denoted by $V_h^{(2)}$. Here we have the property that $V_{h} \subset V_h^{(2)}$. We say $ V_h^{(2)}$ is the (hierarchical) $p$-refined finite element space of $V_h$.
The corresponding discretized problems read as: Find $\vec{u}_h \in V_h \cap V_{BC}$, $\vec{u}_{\frac{h}{2}}\in V_{\frac{h}{2}}\cap V_{BC}$ and $\vec{u}_h^{(2)}\in V_h^{(2)}\cap V_{BC}$ such that 
		\begin{eqnarray*}
			A(\vec{u}_h)(\vec{v}_h)&=&0  \qquad \forall \vec{v}_h \in V_h \cap V_0, \\
				A(\vec{u}_{\frac{h}{2}})(\vec{v}_{\frac{h}{2}})&=&0  \qquad \forall \vec{v}_{\frac{h}{2}} \in V_{\frac{h}{2}} \cap V_0, \\
				A(\vec{u}_h^{(2)})(\vec{v}_h^{(2)})&=&0  \qquad \forall \vec{v}_h^{(2)} \in V_h^{(2)} \cap V_0.
		\end{eqnarray*}

\begin{remark}
	We would like to mention that the domain $\Omega$ is not of polygonal shape.
	%and therefore 
	Therefore, a decomposition into quadrilateral elements is not possible. 
	However, we approximate the  ball  $B$ by a polygonal domain, which is adapted after every refinement process  by describing it as a spherical manifold in \verb|deal.II| \cite{dealII90} 
	using the 
	%\verb|deal.II| \cite{dealII90} 
	command \verb|Triangulation::set_manifold| .
%and the class  \verb|SphericalManifold|. The boundary conditions are ment to be imposed on the approximate boundary.
	\end{remark}

%%%%%%%%%%%%%%%%%%%%%%%%%%%%%%%%%%%%%%%%%%%%%%%%%%
%\section{The Dual Weighted Residual Method and Error Representation}
\section{Dual Weighted Residual Method and Error Representation}
%	As already mentioned in Section~\ref{sec: Introduction} we are not interested in the solution
%	\textcolor{blue}{Philipp: (primal) solution variables instead?} 
%	\bernhard{hier wuerde ich solution nehmen. Es gibt noch kein primal und die soultion besteht aus u und p.} 
%	of our model problem. 
%	However 
We are  primarily interested in one or more particular quantities  of interest. 
	%To estimate the error in these quantities,
	%of interest,  
	%we employ the dual weighted residual (DWR) method \cite{BeRa01}.
	We employ the dual weighted residual (DWR) method \cite{BeRa01} 
	for estimating the error in these quantities.
	To connect  the  quantity of interest $J$ with the model problem, we consider the adjoint problem.
	\subsection{The Adjoint Problem}
	The adjoint problem 
	%is given by:
	reads as follows:
	Find $\vec{z} \in V_0$ such that 
	%for all $\vec{v} \in V_0$ holds
	\begin{eqnarray}
	\label{Eqn:adjointproblem}
	A'(\vec{u})(\vec{v},\vec{z})=J'(\vec{u})(\vec{v})
	\quad \forall \vec{v} \in V_0,
	\end{eqnarray} 
	where $A'$ and $J'$ denote the Frechet derivative of $A$ and $J$, respectively, and $\vec{u}$ is the solution of the model problem  (\ref{Equation: Model Problem}).
%	\subsection{An Error Representation}
	\begin{theorem}\label{Theorem: Error Representation}
		Let us assume that $J \in \mathcal{C}^3(V_{BC},\mathbb{R})$. 
		If $\vec{u}$ solves the model problem (\ref{Equation: Model Problem})
		and $\vec{z}$ solves the adjoint problem  
		(\ref{Eqn:adjointproblem}),
		%for $\vec{u}$, 
		then,  for arbitrary fixed  $\vec{\tilde{u}} \in V_{BC}$ and $ \vec{\tilde{z}} \in V_0$, the following error representation formula holds:
		
		\begin{eqnarray*} \label{Error Representation}
		J(\vec{u})-J(\vec{\tilde{u}})&= \frac{1}{2}\rho(\vec{\tilde{u}})(\vec{z}-\vec{\tilde{z}})+\frac{1}{2}\rho^*(\vec{\tilde{u}},\vec{\tilde{z}})(\vec{u}-\vec{\tilde{u}}) 
		+ \rho (\tilde{\vec{u}})(\vec{\tilde{z}}) + \mathcal{R}^{(3)},\nonumber
		\end{eqnarray*}
	 where
		$\rho(\vec{\tilde{u}})(\cdot) := -A(\vec{\tilde{u}})(\cdot)$,
		$\rho^*(\vec{\tilde{u}},\vec{\tilde{z}})(\cdot) := J'(\vec{\tilde{u}})(\cdot)-A'(\vec{\tilde{u}})(\cdot,\vec{\tilde{z}})$,
		and 
				{\footnotesize
			\begin{equation}
			\label{Error Estimator: Remainderterm}
			\mathcal{R}^{(3)}:=\frac{1}{2}\int_{0}^{1}[J'''(\vec{\tilde{u}}+s\vec{e})(\vec{e},\vec{e},\vec{e})
			-A'''(\vec{\tilde{u}}+s\vec{e})(\vec{e},\vec{e},\vec{e},\vec{\tilde{z}}+s\vec{e^*})
			-3A''(\vec{\tilde{u}}+s\vec{e})(\vec{e},\vec{e},\vec{e^*})]s(s-1)\,ds, 
			\end{equation}
		}
		with $\vec{e}=\vec{u}-\vec{\tilde{u}}$ and $\vec{e^*} =\vec{z}-\vec{\tilde{z}}$.	
	\end{theorem}
	\begin{proof}
		We refer the reader to \cite{EnLaWi18} and \cite{RanVi2013}.
	\end{proof}
		\begin{remark}
			%The arbitrary elements  $\vec{\tilde{u}} \in V_{BC}$ and $ \vec{\tilde{z}} \in V_0$ are meant to be some finite element approximations $\vec{u}_h$ and $\vec{z}_h$. We would like to mention, that we do not require that $\vec{u}_h$ and $\vec{z}_h$ are solutions of the finite dimensional problem.\\
			In practice, the arbitrary elements  $\vec{\tilde{u}} \in V_{BC}$ and $ \vec{\tilde{z}} \in V_0$ will be replaced by
			approximations $\vec{u}_h$ and $\vec{z}_h$ to 
			the corresponding finite element solutions.
		\end{remark}
	\begin{remark}
		The error representation formula in Theorem~\ref{Theorem: Error Representation} is exact but not computable,
		because $\vec{u}$ and $\vec{z}$ are not known.
		%. We do not know $\vec{u}$ and $\vec{z}$.
	\end{remark}

\subsection{Error Estimation and Adaptive Algorithm} \label{SubSection: Error Estimation}
	The different error estimator parts are discussed in  \cite{2019Two}. In particular, it turns out that $\eta_h:= \frac{1}{2}\rho(\vec{\tilde{u}})(\vec{z}-\vec{\tilde{z}})+\frac{1}{2}\rho^*(\vec{\tilde{u}},\vec{\tilde{z}})(\vec{u}-\vec{\tilde{u}}) $ is related to the discretization error \cite{RanVi2013,EnLaWi18,2019Two}. 
	The idea is to replace the quantities $\vec{u}-\vec{\tilde{u}}$  and  $\vec{z}-\vec{\tilde{z}}$ by some computable quantities.  
	This can be done via higher order interpolation \cite{BeRa01,RanVi2013} or hierarchically (via an additional solve on an enriched space) \cite{BeRa01,EnLaWi18,KoeBruBau2019a}. 
	If $\vec{u^+_h}$, $\vec{z^+_h}$ are the solution, then we approximate $\vec{u}-\vec{\tilde{u}}$ and $\vec{z}-\vec{\tilde{z}}$ by $\vec{u^+_h}-\vec{\tilde{u}}$ and $\vec{z^+_h}-\vec{\tilde{z}}$, respectively. 
	The new computable error estimator then reads as $$\eta_h^+:= \frac{1}{2}\rho(\vec{\tilde{u}})(\vec{z^+_h}-\vec{\tilde{z}})+\frac{1}{2}\rho^*(\vec{\tilde{u}},\vec{\tilde{z}})(\vec{u^+_h}-\vec{\tilde{u}}). $$
	%In \cite{2019Two} it was shown under some saturation assumption 
	Under some saturation assumption, it was shown in \cite{2019Two} that
	the resulting error estimator is efficient and reliable. 
	We consider the two different error estimators
	\begin{eqnarray*}
	\eta_h^{(2)}&:= \frac{1}{2}\rho(\vec{\tilde{u}})(\vec{z}_h^{(2)}-\vec{\tilde{z}})+&\frac{1}{2}\rho^*(\vec{\tilde{u}},\vec{\tilde{z}})(\vec{u}_h^{(2)}-\vec{\tilde{u}}), \\
	\eta_{\frac{h}{2}}&:= \frac{1}{2}\rho(\vec{\tilde{u}})(\vec{z}_{\frac{h}{2}}-\vec{\tilde{z}})+&\frac{1}{2}\rho^*(\vec{\tilde{u}},\vec{\tilde{z}})(\vec{u}_{\frac{h}{2}}-\vec{\tilde{u}}).
	\end{eqnarray*}
	%We say $\eta_h^{(2)}$ is the $p$-enriched and $\eta_{\frac{h}{2}}$ is the $h$-enriched error estimator.  
	We call $\eta_h^{(2)}$ and $\eta_{\frac{h}{2}}$ the $p$-enriched and $h$-enriched error estimators, respectively.
	The error estimators are localized using the partition of unity technique proposed in \cite{RiWi15_dwr}.
	The marking strategy and algorithms are the same as in \cite{2019Two}.
	\begin{remark}
		The efficiency and reliability are not guaranteed  under the corresponding saturation assumption in \cite{2019Two} for $\eta_{\frac{h}{2}}$, since the boundary is  adapted in every refinement step.
	\end{remark}
	\begin{remark}
			We use the algorithm presented in \cite{2019Two}. The algorithm using $p$ enrichment coincides with Algorithm 3 in \cite{2019Two}. In the algorithm, where we use $h$ enrichment, we replace $V_h^{(2)}$ by $V_{\frac{h}{2}} $. 
		\end{remark}

%%%%%%%%%%%%%%%%%%%%%%%%%%%%%%%%%%%%%%%%%%%%%%%%%%%5	
\section{Numerical Experiment}
%
%In this experiment we 
We compare the two error estimators introduced in Section~\ref{SubSection: Error Estimation}. In the $p$ enriched case,  we use uniform $p$ refinement for the hierarchical approximation. The results for $p$ enrichment have already been computed in \cite{2019Two}.  In the $h$ enriched case, we use uniform $h$ refinement.  The configuration of the problem is given in Section~\ref{SubSection: Model Problem}.

\subsection{Quantities of Interest}
%As quantities of interest we 
We use 
%\textcolor{blue}{Philipp: those defined in instead?} \bernhard{Das wuerde ich so lassen.}
the quantities of interest defined in \cite{TurSchabenchmark1996,2019Two}:
			\begin{eqnarray*}
			\Delta p(\vec{u}):= &p(X_1)-p(X_2),\\
			c_{\text{drag}}(\vec{u}):= &C\int_{\partial B} \left[\nu (\nabla u + \nabla u^T)\vec{n} - p \vec{n }\right]\cdot \vec{e}_1\,\text{ d}s_{(x,y)},\\	
			c_{\text{lift}}(\vec{u}):= &C\int_{\partial B} \left[\nu (\nabla u + \nabla u^T)\vec{n} - p \vec{n }\right]\cdot \vec{e}_2\,\text{ d}s_{(x,y)},\\	
			\end{eqnarray*}
			where $C = 500$, $X_1= (0.15,0.2)$, $X_2= (0.25,0.2)$, $\vec{e}_1:=(1,0)$, $\vec{e}_2:=(0,1)$, and $\vec{n}$ denotes the outer normal vector.
			To do adaptivity for all of them at  once we combine them to one functional $$J_\mathfrak{E}(\vec{v}_h):=\frac{|\Delta p(\vec{u}^+_h-\vec{v}_h)|}{|\Delta p(\vec{u}_h)|}+\frac{|c_{\text{drag}}(\vec{u}^+_h-\vec{v}_h)|}{|c_{\text{drag}}(\vec{u}_h)|}+\frac{|c_{\text{lift}}(\vec{u}^+_h-\vec{v}_h)|}{|c_{\text{lift}}(\vec{u}_h)|}.$$
By $J_\mathfrak{E}^p$ or $J_\mathfrak{E}^h$, we denote the functionals where we replace $\vec{u}^+_h$ with  $\vec{u}_h^{(2)}$ or $\vec{u}_\frac{h}{2}$,
respectively.
			More information on how to treat multiple functionals at once can be found in \cite{HaHou03,Ha08, BruZhuZwie16,PARDO20101953,KerPruChaLaf2017,EnWi17,EnLaWi18,2019Two}. 
			The implementation is done in the finite element library \verb|deal.II| \cite{dealII90}, and follows the code in \cite{2019Two}. In this section, we compare two different sequences of meshes. 
			The sequences are generated by the error estimators $\eta_h^{(2)}$ and $\eta_{\frac{h}{2}}$.
			First of all, let us define the effectivity indices by 
			$$I_{eff}^p := \frac{\eta_h^{(2)}}{|J_\mathfrak{E}^p(u)-J_\mathfrak{E}^p(u_h)|} \qquad \text{ and } \qquad I_{eff}^h := \frac{\eta_{\frac{h}{2}}}{|J_\mathfrak{E}^h(u)-J_\mathfrak{E}^h(u_h)|}.$$
			The $p$ enriched discrete remainder part of the error estimator $\eta^{(2)}_\mathcal{R}$ is defined as the quantity \eqref{Error Estimator: Remainderterm}, where we replace $\vec{\tilde{u}},\vec{\tilde{z}},\vec{u},\vec{z}$ by $\vec{u}_h,\vec{z}_h,\vec{u}_h^{(2)},\vec{z}_h^{(2)}$, respectively. The $h$ enriched discrete remainder part of the error estimator $\eta_{\mathcal{R},\frac{h}{2}}$ is defined as the quantity \eqref{Error Estimator: Remainderterm}, where we replace $\vec{\tilde{u}},\vec{\tilde{z}},\vec{u},\vec{z}$ by $\vec{u}_h,\vec{z}_h,\vec{u}_\frac{h}{2},\vec{z}_\frac{h}{2}$, respectively.
			Finally, we define the gaps between the theoretical findings in \cite{2019Two} by
			$$\eta_{\mathbb{E}}^{(2)} := \left||J(\vec{u}_h^{(2)})-J(\vec{u}_h)| - |\eta_h^{(2)}+ \rho (\vec{u}_h,\vec{z}_h) +\eta_{\mathcal{R}}^{(2)} |\right|,$$
			and
			$$\eta_{\mathbb{E},\frac{h}{2}} :=\left| |J(\vec{u}_\frac{h}{2})-J(\vec{u}_h)| - |\eta_\frac{h}{2}+ \rho (\vec{u}_h,\vec{z}_h) +\eta_{\mathcal{R},\frac{h}{2}} |\right|.$$

%%%%%%%%%%%%%%%%%%%%%%%%%%%%%%%%%%%%%%%%%%%%%%%%%%%%%%%
\subsection{Discussion of the Results}

In Figure \ref{Figure: IeffpIeffh}, the effectivity indices for the two different types of error estimators are shown on their respective grids. 
We see that $h$-enrichment delivers effectivity indices which are very close to one, whereas for $p$-enrichment we have effectivity indices in the range of  $0.2-8.1$. This was also observed in \cite{2019Two}.
In the case of $p$-enrichment, the saturation assumption is violated multiple times, as we observe in Figure~\ref{Figure: Saturation}. 
%\red{
The saturation assumtion is violated if the error
$|J_\mathfrak{E}^p(\vec{u}_h^{(2)})-J_\mathfrak{E}^p(\vec{u})|$
in the enriched solution 
is larger than $|J_\mathfrak{E}^p(\vec{u}_h)-J_\mathfrak{E}^p(\vec{u})|$.
%} 
In the case of $h$-enrichment, this  always happens.
%is always the case.   
If we compare the errors of the single functionals, which are monitored in Figure \ref{Figure: lift}, Figure \ref{Figure: drag} and Figure \ref{Figure: p}, we conclude that the meshes generated by the $p$-enriched error estimator lead to %better
smaller
errors in the single functionals. 
If all the conditions in \cite{2019Two} are fulfilled, then $\eta_{\mathbb{E}}^{(2)}$ and $\eta_{\mathbb{E},\frac{h}{2}} $ are zero. 
However, in the computation of the error estimators, our overall round-off error is in the order of $\varepsilon(\text{double}) \times \text{DOFs}$, where $\varepsilon(\text{double})= 2^{-52}$ is the machine precision for double 
%precision 
floating point numbers%
% in\verb|C++| 
\footnote{\url{https://en.wikipedia.org/wiki/Machine_epsilon}}.  In the case of $p$ enrichment, we observe in Figure~\ref{Figure: Theoretical Error} that $\eta_{\mathbb{E}}^{(2)}$ indeed is in the order or even better than the round off errors when summing up the different error contributions. 
In this case, all requirements
%\textcolor{blue}{Philipp: of the saturation assumption?} \bernhard{Hier nicht. Hier ist eine subspace property gemeint.}
 are fulfilled. For $h$ enrichment, we do not have the inclusion $V_h \subset V_\frac{h}{2}$ due to the geometrical approximation. Therefore, these conditions are violated. The effects  are monitored in Figure~\ref{Figure: Theoretical Error} as well. The quantity $\eta_{\mathbb{E},\frac{h}{2}} $ does not only contain numerical round off errors, but also errors coming from the geometrical approximation. However, this is a non-local quantity, and 
 the localization is not straightforward.
{\tiny
\begin{figure}[H]
	\begin{minipage}[t]{0.45\linewidth}
		\centering
			\ifMAKEPICS
			\begin{gnuplot}[terminal=epslatex]
				set output "Figures/Example1tex.tex"
				set datafile separator "|"
				set logscale x
				set grid ytics lc rgb "#bbbbbb" lw 1 lt 0
				set grid xtics lc rgb "#bbbbbb" lw 1 lt 0
				set xlabel '\text{DOFs (degrees of freedom)}'
				set format '%g'
				plot '< sqlite3 Data/p_enrichment/NS2D_1_adaptive_Lift_Drag_PDiff_Radon1_2.db "SELECT DISTINCT DOFS_primal, Ieff from data"' u 1:2 w lp lw 2 title '$I_{eff}^p$',\
				'< sqlite3 Data/h_enrichment/data.db "SELECT DISTINCT DOFS_primal, Ieff from data"' u 1:2 w lp lw 2 title '$I_eff^h$',\
				1lw 4 dashtype 2 title '$1$'\
				#'< sqlite3 Data/p_enrichment/NS2D_1_adaptive_Lift_Drag_PDiff_Radon1_2.db "SELECT DISTINCT DOFS_primal , Estimated_Error_primal from data"' u 1:2 w lp lw 2 title 'Estimated Error(primal)',\
				#'< sqlite3 Data/p_enrichment/NS2D_1_adaptive_Lift_Drag_PDiff_Radon1_2.db "SELECT DISTINCT DOFS_primal , Estimated_Error_adjoint from data"' u 1:2 w lp lw 2 title 'Estimated(adjoint)',\
			\end{gnuplot}
			\fi
			\scalebox{0.47}{\tiny\input{Figures/Example1tex.tex}}
			\captionof{figure}{The two effectivity indices on the corresponding meshes. \label{Figure: IeffpIeffh}}
				\end{minipage} \hfill
								\begin{minipage}[t]{0.45\linewidth}
										\centering
									\ifMAKEPICS
									\begin{gnuplot}[terminal=epslatex]
										set output "Figures/Example3tex.tex"
										set datafile separator "|"
										set logscale
										set grid ytics lc rgb "#bbbbbb" lw 1 lt 0
										set grid xtics lc rgb "#bbbbbb" lw 1 lt 0
										set xlabel '\text{DOFs}'
										set ylabel '\text{Error}'
										set format '%g'
										plot '< sqlite3 Data/p_enrichment/NS2D_1_adaptive_Lift_Drag_PDiff_Radon1_2.db "SELECT DISTINCT DOFS_primal, Exact_Error from data"' u 1:2 w lp lw 2 title '$|J_\mathfrak{E}^p(\vec{u}_h)-J_\mathfrak{E}^p(\vec{u})|$',\
										'< sqlite3 Data/h_enrichment/data.db "SELECT DISTINCT DOFS_primal, Exact_Error from data"' u 1:2 w lp lw 2 title '$|J_\mathfrak{E}^h(\vec{u}_h)-J_\mathfrak{E}^h(\vec{u})|$',\
										'< sqlite3 Data/p_enrichment/NS2D_1_adaptive_Lift_Drag_PDiff_Radon1_2.db "SELECT DISTINCT DOFS_primal, abs(Ju-Juh2) from data"' u 1:2 w lp lw 2 title '$|J_\mathfrak{E}^p(\vec{u}_h^{(2)})-J_\mathfrak{E}^p(\vec{u})|$',\
										'< sqlite3 Data/h_enrichment/data.db "SELECT DISTINCT DOFS_primal, abs(Ju-Juh2) from data"' u 1:2 w lp lw 2 title '$|J_\mathfrak{E}^h(\vec{u}_{\frac{h}{2}})-J_\mathfrak{E}^h(\vec{u})|$',\
										#'< sqlite3 Data/h_enrichment/data.db "SELECT DISTINCT DOFS_primal , Estimated_Error_primal from data"' u 1:2 w lp lw 2 title 'Estimated Error(primal)',\
										#'< sqlite3 Data/h_enrichment/data.db "SELECT DISTINCT DOFS_primal , Estimated_Error_adjoint from data"' u 1:2 w lp lw 2 title 'Estimated(adjoint)',\
									\end{gnuplot}
									\fi
									\scalebox{0.47}{\tiny\input{Figures/Example3tex.tex}}
									\captionof{figure}{Errors in $J_\mathfrak{E}^p$ and $J_\mathfrak{E}^h$ in the solution and enriched solution. \label{Figure: Saturation}}
								\end{minipage}
		\begin{minipage}[t]{0.45\linewidth}
				\centering
					\ifMAKEPICS
					\begin{gnuplot}[terminal=epslatex]
						set output "Figures/Example4tex.tex"
						set datafile separator "|"
						set logscale
						set grid ytics lc rgb "#bbbbbb" lw 1 lt 0
						set grid xtics lc rgb "#bbbbbb" lw 1 lt 0
						set xlabel '\text{DOFs}'
						set format '%g'
						plot '< sqlite3 Data/p_enrichment/NS2D_1_adaptive_Lift_Drag_PDiff_Radon1_2.db "SELECT DISTINCT DOFS_primal, relativeError2 from data"' u 1:2 w lp lw 2 title '$c_{\text{\text{lift}}}^p$',\
						'< sqlite3 Data/h_enrichment/data.db "SELECT DISTINCT DOFS_primal, relativeError1 from data"' u 1:2 w lp lw 2 title '$c_{\text{\text{lift}}}^h$',\
							'< sqlite3 Data/Uniform/uniform_data_p.db "SELECT DISTINCT DOFS_primal, relativeError2 from data_global"' u 1:2 w lp lw 2 title '$c_{\text{\text{lift}}}$',\
						#'< sqlite3 Data/h_enrichment/data.db "SELECT DISTINCT DOFS_primal , Estimated_Error_primal from data"' u 1:2 w lp lw 2 title 'Estimated Error(primal)',\
						#'< sqlite3 Data/h_enrichment/data.db "SELECT DISTINCT DOFS_primal , Estimated_Error_adjoint from data"' u 1:2 w lp lw 2 title 'Estimated(adjoint)',\
					\end{gnuplot}
					\fi
					\scalebox{0.47}{\tiny\input{Figures/Example4tex.tex}}
					\captionof{figure}{The errors in  $c_{\text{lift}}$ for  refinement with $p$-enriched error estimation ($c_{\text{\text{lift}}}^p$), refinement with $h$-enriched error estimation ($c_{\text{\text{lift}}}^h $), and uniform refinement ($c_{\text{\text{lift}}}$).\label{Figure: lift}}
											\end{minipage} \hfill
											\begin{minipage}[t]{0.45\linewidth}
													\centering
					\ifMAKEPICS
					\begin{gnuplot}[terminal=epslatex]
						set output "Figures/Example5tex.tex"
						set datafile separator "|"
						set logscale
						set grid ytics lc rgb "#bbbbbb" lw 1 lt 0
						set grid xtics lc rgb "#bbbbbb" lw 1 lt 0
						set xlabel '\text{DOFs}'
						set format '%g'
						plot '< sqlite3 Data/p_enrichment/NS2D_1_adaptive_Lift_Drag_PDiff_Radon1_2.db "SELECT DISTINCT DOFS_primal, relativeError1 from data"' u 1:2 w lp lw 2 title '$c_{\text{\text{drag}}}^p$',\
						'< sqlite3 Data/h_enrichment/data.db "SELECT DISTINCT DOFS_primal, relativeError2 from data"' u 1:2 w lp lw 2 title '$c_{\text{\text{drag}}}^h$',\
							'< sqlite3 Data/Uniform/uniform_data_p.db "SELECT DISTINCT DOFS_primal, relativeError1 from data_global"' u 1:2 w lp lw 2 title '$c_{\text{\text{drag}}}$',\
						#'< sqlite3 Data/h_enrichment/data.db "SELECT DISTINCT DOFS_primal , Estimated_Error_primal from data"' u 1:2 w lp lw 2 title 'Estimated Error(primal)',\
						#'< sqlite3 Data/h_enrichment/data.db "SELECT DISTINCT DOFS_primal , Estimated_Error_adjoint from data"' u 1:2 w lp lw 2 title 'Estimated(adjoint)',\
					\end{gnuplot}
					\fi
					\scalebox{0.47}{\tiny\input{Figures/Example5tex.tex}}
					\captionof{figure}{The errors in  $c_{\text{drag}}$ for  refinement with $p$-enriched error estimation ($c_{\text{\text{drag}}}^p$), refinement with $h$-enriched error estimation ($c_{\text{\text{drag}}}^h $), and uniform refinement ($c_{\text{\text{drag}}}$). \label{Figure: drag}}
						\end{minipage} 
						\begin{minipage}[t]{0.45\linewidth}
								\centering
					\ifMAKEPICS
					\begin{gnuplot}[terminal=epslatex]
						set output "Figures/Example6tex.tex"
						set datafile separator "|"
						set logscale
						set grid ytics lc rgb "#bbbbbb" lw 1 lt 0
						set grid xtics lc rgb "#bbbbbb" lw 1 lt 0
						set xlabel '\text{DOFs}'
						set format '%g'
						plot '< sqlite3 Data/p_enrichment/NS2D_1_adaptive_Lift_Drag_PDiff_Radon1_2.db "SELECT DISTINCT DOFS_primal, relativeError0 from data"' u 1:2 w lp lw 2 title '$\Delta p^p$',\
						'< sqlite3 Data/h_enrichment/data.db "SELECT DISTINCT DOFS_primal, relativeError0 from data"' u 1:2 w lp lw 2 title '$\Delta p^h$',\
						'< sqlite3 Data/Uniform/uniform_data_p.db "SELECT DISTINCT DOFS_primal, relativeError0 from data_global"' u 1:2 w lp lw 2 title '$\Delta p$',\
						#'< sqlite3 Data/h_enrichment/data.db "SELECT DISTINCT DOFS_primal , Estimated_Error_primal from data"' u 1:2 w lp lw 2 title 'Estimated Error(primal)',\
						#'< sqlite3 Data/h_enrichment/data.db "SELECT DISTINCT DOFS_primal , Estimated_Error_adjoint from data"' u 1:2 w lp lw 2 title 'Estimated(adjoint)',\
					\end{gnuplot}
					\fi
					\scalebox{0.47}{\tiny\input{Figures/Example6tex.tex}}
					\captionof{figure}{The errors in  ${\Delta p}$ for  refinement with $p$-enriched error estimation (${\Delta p}^p$), refinement with $h$-enriched error estimation (${\Delta p}^h $), and uniform refinement (${\Delta p}$).\label{Figure: p}}
						\end{minipage} \hfill
										\begin{minipage}[t]{0.45\linewidth}
												\centering
											\scalebox{0.47}{\tiny\input{Figures/Example2tex.tex}}
											\captionof{figure}{The remainder parts $\eta^{(2)}_\mathcal{R}$, $\eta_{\mathcal{R},\frac{h}{2}}$ and gap parts $\eta_{\mathbb{E}}^{(2)}$, $\eta_{\mathbb{E},\frac{h}{2}} $ for $p$ and $h$ enrichment. The remainder parts $\eta^{(2)}_\mathcal{R}$, $\eta_{\mathcal{R},\frac{h}{2}}$ are indeed neglectable as usually done in literature. \label{Figure: Theoretical Error}}
										\end{minipage} %
\end{figure}% 
}

\section*{Acknowledgement}
 This work has been supported by the Austrian Science Fund (FWF) under the grant
 P 29181
 `Goal-Oriented Error Control for Phase-Field Fracture Coupled to Multiphysics Problems'. Furthermore, we thank Daniel Jodlbauer for discussions.

 \bibliographystyle{abbrv}
 \bibliography{lit.bib}
\end{document}